\documentclass[final,10pt]{siamltex}
\usepackage{enumerate}
\usepackage{amsmath,amssymb,graphicx,float,epsf}

\def\NN{\hbox{I\kern-.2em\hbox{N}}}
\def\RR{\hbox{I\kern-.2em\hbox{R}}}

\numberwithin{equation}{section}

\newtheorem{remark}{Remark}
\newtheorem{corol}{Corollary}
\newtheorem{example}{Example}

\title{New Global Exponential Stability Criteria for Nonlinear Delay Differential Systems
with Applications to BAM Neural Networks
\thanks{The research of the second author was
partially supported by NSERC.}}

\author{Leonid Berezansky,
\thanks{Dept. of Math, Ben-Gurion University of the Negev, Beer-Sheva 84105, Israel}
\and
Elena Braverman
\thanks{Dept. of Math \& Stats, University of Calgary, 2500 University Dr. NW, Calgary, AB, Canada T2N1N4}
\and Lev Idels
\thanks{Dept. of Math, Vancouver Island University (VIU),
   900 Fifth St. Nanaimo, BC, Canada  V9S5J5}}
\begin{document}
\maketitle

%\footnotetext[1]{Corresponding author. E-mail {\em
%maelena@math.ucalgary.ca}. Fax (403)-282-5150. Phone (403)-220-3956.
%Partially supported by the NSERC Research Grant.}

\begin{abstract}
We consider a nonlinear non-autonomous system with
time-varying delays
$$
\dot{x_i}(t)=-a_i(t)x_{i}(h_i(t))+\sum_{j=1}^mF_{ij}(t,x_j(g_{ij}(t)))
$$
which has a large number of applications in the theory of artificial neural networks.
Via the M-matrix method, easily verifiable sufficient stability conditions for the nonlinear system and its linear version are
obtained.
Application of the main theorem requires just to check whether a  matrix, which is explicitly constructed
by the system's parameters, is an $M$-matrix.
Comparison with the tests obtained by K. Gopalsamy (2007) and B. Liu (2013)
for BAM neural networks illustrates novelty of the stability theorems.
Some open problems conclude the paper.
\end{abstract}



\pagestyle{myheadings}
\thispagestyle{plain}
\markboth{Exponential Stability for BAM Neural Networks}{BEREZANSKY, BRAVERMAN, IDELS}

\section{Introduction}

One of the main motivations to study the nonlinear delay differential system
\begin{equation}\label{0}
\dot{x_i}(t)=-a_i(t)x_{i}(h_i(t))+\sum_{j=1}^mF_{ij}\left( t,x_j(g_{ij}(t)) \right),~t\geq 0,~ i=1,\dots,m
\end{equation}
and its linear version
\begin{equation}\label{01}
\dot{x_i}(t)=\sum_{j=1}^m a_{ij}(t)x_j(g_{ij}(t)), ~ i=1,\dots,m,
\end{equation}
is their importance in the study of artificial neural network models \cite{Ha,K3}.

For linear system (\ref{01}) several very interesting results were obtained in \cite{Dib,Gy,SoTangZou,SoTangZou2}.
In \cite{Gy} system (\ref{01})  with constant coefficients $a_{ij}$ was examined;
in \cite{SoTangZou2} the proofs were based on the assumption that $a_{ij}$ and $g_{ij}$ are continuous functions
and $|a_{ij}(t)|\leq \beta_{ij}a_{ii}(t)$.
Most of the results for system (\ref{0}) %and mathematical models described by (\ref{0})
were obtained in the case $h_i(t)\equiv t$ (see, for example, \cite{Ar}). Also the requirement that all the functions involved in
the system are continuous seem unduly restrictive, and we relax this assumption.
In the present paper, we consider the so-called pure-delay case $h_i(t)\not\equiv t$,
assuming that all parameters are measurable functions, and $F_{ij}(t,u)$ are Caratheodory functions.
Via M-matrix Method we obtain novel stability results for nonlinear non-autonomous system \eqref{0}
and linear non-autonomous system \eqref{01}.
It is to be emphasized that our technique does not require a long sequence of other theorems
or conditions that must be proven or cited before the main result is justified.

Gopalsamy in \cite{Gop} studied a model of networks known as Bidirectional Associative Memory (BAM) with leakage delays:
\begin{equation}\label{02}
\begin{array}{l}
{\displaystyle \frac{dx_i(t)}{dt}=-a_ix_i\left(t-\tau_i^{(1)}\right)+\sum_{j=1}^n a_{ij}f_j\left( y_j\left( t-\sigma_j^{(2)} \right) \right)+I_i}\\
{\displaystyle \frac{dy_i(t)}{dt}=-b_iy_i\left(t-\tau_i^{(2)}\right)+\sum_{j=1}^n b_{ij}g_j\left( x_j \left( t-\sigma_j^{(1)}\right) \right)+J_i}
\end{array} ~~i=1,\dots,n.
\end{equation}
Here $\tau^{(k)}_{i}, \sigma^{(k)}_{j}~(k=1,2)$ are the leakage and the transmission delays accordingly.
In \cite{Gop} sufficient conditions for the existence of a unique equilibrium and its global stability for system  (\ref{02}) were obtained.
Some interesting results for system \eqref{02} were obtained via the construction of Lyapunov functionals in  \cite{Chen,Liu,Wa,Wang2,Zen,Zen2}.

To extend and improve the results obtained  in \cite{Gop,Liu,Peng}, we  apply our main theorem to  the non-autonomous system
\begin{equation}\label{03}
\begin{array}{l}
{\displaystyle \frac{dx_i(t)}{dt}=r_i(t)\left[-a_ix_i\left( h_i^{(1)}(t) \right)+\sum_{j=1}^n
a_{ij}f_j\left(y_j \left( l_j^{(2)}(t) \right) \right)+I_i\right] } \\
{\displaystyle \frac{dy_i(t)}{dt}=p_i(t)\left[-b_iy_i\left( h_i^{(2)}(t) \right)+\sum_{j=1}^n
b_{ij}g_j\left( x_j\left( l_j^{(1)}(t) \right) \right)+J_i\right] }.
\end{array}
\end{equation}
Let us quickly sketch what we accomplish here.
Section \ref{mainresults} incorporates the main result of the paper: if a certain matrix which is  explicitly constructed from the
functions and the coefficients of the system is an $M$-matrix, then the system is globally exponentially stable.
It is demonstrated that the stability condition for a nonlinear system of two equations with constant delays
improves the test obtained in \cite{Gop}.
In Section \ref{BAMsection} we examine stability of BAM models and obtain stability results that
for a nonlinear BAM systems generalize the main theorem in \cite{Liu}.
Finally, Section~\ref{discussion} contains discussion and outlines some open problems.

\section{Main Results}
\label{mainresults}
Consider for any $t_0\geq 0$ the system of delay differential equations
\begin{equation}\label{1}
\dot{x_i}(t)=-a_i(t)x_{i}(h_i(t))+\sum_{j=1}^mF_{ij}\left( t,x_j(g_{ij}(t)) \right),~t\geq t_0,~ i=1,\dots,m,
\end{equation}
with the initial conditions
\begin{equation}\label{2}
x_i(t)=\varphi_i(t),~ t\leq t_0,~ i=1,\dots,m,
\end{equation}
under the following assumptions:
\\
(a1) $a_i$ are Lebesgue measurable essentially bounded on $[0,\infty)$ functions, $0<\alpha_i\leq a_i(t)\leq A_i$;\\
(a2) $F_{ij}(t,\cdot)$ are continuous functions, $F_{ij}(\cdot, u)$ are measurable locally essentially bounded functions,
$|F_{ij}(t,u)|\leq L_{ij}|u|, t\geq 0$;\\
(a3) $h_i, g_{ij}$ are measurable functions, $0\leq t-h_i(t)\leq \tau_i,~ 0\leq t-g_{ij}(t)\leq \sigma$;\\
(a4) $\varphi_i$ are continuous functions on $[t_0-\sigma, t_0]$, where \\
$\sigma=\max\{\tau_k,\sigma_{ij}, k,i,j=1,\dots, m\}$.

Henceforth  assume that conditions $(a1)-(a4)$ hold for  problem (\ref{1}), (\ref{2}) and
its modifications, and the problem has a unique solution.\\
We will use some traditional notations.
A matrix $B=(b_{ij})_{i,j=1}^m$ is nonnegative if $b_{ij} \geq 0$ and positive if $b_{ij}>0$, $i,j=1, \dots,m$;  $\|a\|$ is a norm of a column vector
$a=(a_1,\dots,a_m)^T$ in $\RR^m$; $\|B\|$ is the corresponding matrix norm of a matrix $B$, $|a|=(|a_1|,\dots,|a_m|)^T$ and $|B|=(|b_{ij}|)_{i,j=1}^m$.
As usual, function $X(t)=(x_1(t),\dots,x_m(t))^T$
is a solution of  (\ref{1}), (\ref{2}) if it satisfies (\ref{1}) almost everywhere
for $t>t_0$ and (\ref{2}) for $t\leq t_0$. Problem (\ref{1}),(\ref{2})
has a unique global solution on $[t_0,\infty)$, if, for example, we assume  along with $(a1)-(a4)$ that functions $F_{ij}(t,u)$ are locally
Lipschitz in $u$.
The following classical definition of an $M-$matrix will be used.
\begin{definition} \cite{Berman}
A matrix $B=(b_{ij})_{i,j=1}^m$ is called a (nonsingular) {\em $M$-matrix} if $b_{ij}\leq 0, i\neq j$ and
one of the following equivalent conditions holds:\\
- there exists a positive inverse matrix $B^{-1} > 0$;\\
- the principal minors of matrix $B$ are positive.
\end{definition}

\begin{lemma}\cite{Berman}\label{lemma0}
$B$ is an $M-$matrix if $b_{ij}\leq 0, i\neq j$ and at least one of the following conditions holds:\\
- $b_{ii}>\sum_{j\neq i}|b_{ij}|,~ i=1,\dots, m$;\\
- $b_{jj}>\sum_{i\neq j}|b_{ij}|,~ j=1,\dots, m$;\\
- there exist positive numbers $\xi_i, i=1,\dots, m$ such that $\xi_ib_{ii}>\sum_{j\neq i}\xi_j|b_{ij}|,~ i=1,\dots, m$;\\
- there exist positive numbers $\xi_i, i=1,\dots, m$ such that $\xi_jb_{jj}>\sum_{i\neq j}\xi_i|b_{ij}|,~ j=1,\dots, m$.
\end{lemma}

\begin{definition}
System (\ref{1}) is globally exponentially stable if there exist $M>0, \lambda>0$
such that for any solution $X(t)$ of problem (\ref{1}),(\ref{2})
the inequality
$$
\|X(t)\|\leq M e^{-\lambda (t-t_0)}\left( \|x(t_0)\|+\sup_{t<t_0}\|\varphi(t)\| \right)
$$
holds, where $M$ and $\lambda$ do not depend on $t_0$.
\end{definition}
We define matrix $C$  as follows
\begin{equation}\label{3}
C=(c_{ij})_{i,j=1}^m, ~~c_{ii}=1-\frac{A_i(A_i+L_{ii})\tau_i+L_{ii}}{\alpha_i}, ~ c_{ij}=-\frac{A_iL_{ij}\tau_i+L_{ij}}{\alpha_i},i\neq j.
\end{equation}

\begin{theorem}\label{th1}
Suppose  $C$ defined by (\ref{3}) is an M-matrix. Then system (\ref{1}) is globally exponentially stable.
\end{theorem}
\begin{proof}
The solution $X(t)=(x_1(t),\dots,x_n(t))^T$ of problem (\ref{1}),(\ref{2}) is also a solution of the problem
\begin{equation}\label{1c}
\dot{x_i}(t)=-a_i(t)x_{i}(h_i(t))+\sum_{j=1}^m F_{ij}\left( t,x_j(g_{ij}(t))+\varphi_j(g_{ij}(t)))-a_i(t)\varphi_i(h_i(t) \right),~t\geq t_0,
\end{equation}
$i=1,\dots,m$,
where we assume that
$x_{i}(t)=0$, $t<t_0$ and
$\varphi_i(t)=0$ for $t\geq t_0$.
After the substitution
$x_i(t)=e^{-\lambda (t-t_0)}y_i(t), ~t\geq t_0$, where $0<\lambda < \min_{i} \alpha_{i}$, equation (\ref{1c}) has the form
\begin{equation}\label{3a}
\begin{array}{ll}
\dot{y_i}(t)= & {\displaystyle \lambda y_i(t)-e^{\lambda(t-h_i(t))}a_i(t)y_i(h_i(t))}
\\ \\  &{\displaystyle
+ \sum_{j=1}^me^{\lambda (t-t_0)}F_{ij}\left( t,e^{-\lambda (g_{ij}(t)-t_0)} y_j(g_{ij}(t))+\varphi_j(g_{ij}(t)) \right)
} \\ \\  &{\displaystyle -e^{\lambda (t-t_0)}a_i(t)\varphi_i(h_i(t)).}
\end{array}
\end{equation}
After denoting
$
\mu_i(t):= e^{\lambda(t-h_i(t))}a_i(t)-\lambda,
$
equation (\ref{3a}) can be rewritten as
\begin{eqnarray*}
\dot{y_i}(t) & = & -\mu_i(t) y_i(t)+e^{\lambda(t-h_i(t))}a_i(t)\int_{h_i(t)}^t  \dot{y_i}(s)ds \\
& & + \sum_{j=1}^me^{\lambda (t-t_0)}F_{ij}\left( t,e^{-\lambda (g_{ij}(t)-t_0)}y_j(g_{ij}(t))+\varphi_j(g_{ij}(t)) \right)
\\ & &  -e^{\lambda (t-t_0)}a_i(t)\varphi_i(h_i(t)).
\end{eqnarray*}
For $\dot{y_{i}}(s)$ we substitute the right-hand side of equation (\ref{3a})
\begin{eqnarray*}
\dot{y_i}(t) & = &-\mu_i(t) y_i(t)
+ e^{\lambda(t-h_i(t))}a_i(t)\left. \int_{h_i(t)}^t
\right[\lambda y_i(s)-e^{\lambda(s-h_i(s))}a_i(s)y_i(h_i(s))
\\
& & \left. + \sum_{j=1}^me^{\lambda (s-t_0)}F_{ij} \left( s,e^{-\lambda (g_{ij}(s)-t_0)} y_j(g_{ij}(s))+\varphi_j(g_{ij}(s))
\right) -e^{\lambda (s-t_0)}a_i(t)\varphi_i(h_i(s))\right]ds
\\
& & +\sum_{j=1}^me^{\lambda (t-t_0)}F_{ij}\left( t,e^{-\lambda g_{ij}(t)-t_0)} y_j(g_{ij}(t))+\varphi_j(g_{ij}(t)) \right)
-e^{\lambda (t-t_0)}a_i(t)\varphi_i(h_i(t)).
\end{eqnarray*}
Hence
\begin{eqnarray*}
y_i(t) & =  &  \left. e^{-\int_{t_0}^t \mu_i(s)ds} \!\!\! x_i(t_0)+\int\limits_{t_0}^t e^{-\int_s^t \mu_i(\zeta)\,d\zeta}
\left( e^{\lambda(s-h_i(s))}a_i(t) \int\limits_{h_i(s)}^s
\right[\lambda y_i(\zeta)-e^{\lambda(\zeta-h_i(\zeta))}a_i(\zeta)y_i(h_i(\zeta)) \right.
\\
&  & \left. + \sum_{j=1}^m e^{\lambda (\zeta-t_0)} F_{ij} \left( \zeta,e^{-\lambda
(g_{ij}(\zeta)-t_0)} y_j(g_{ij}(\zeta))+\varphi_j(g_{ij}(\zeta))\right)
-e^{\lambda (\zeta-t_0)}a_i(\zeta)\varphi_i(h_i(\zeta))\right]d\zeta
\\
&  &  \left.+\sum_{j=1}^me^{\lambda (s-t_0)}F_{ij}\left( s,e^{- \lambda (g_{ij}(s)-t_0)} y_j(g_{ij}(s))+\varphi_j(g_{ij}(s))
\right) -e^{\lambda (s-t_0)}a_i(s)\varphi_i(h_i(s))\right)ds.
\end{eqnarray*}
Then
\begin{eqnarray*}
|y_i(t)| & \leq & |x_i(t_0)|+\int_{t_0}^t e^{-\int_s^t \mu_i(\zeta)d\zeta}\mu_i(s)
\left(A_ie^{\lambda \tau_i}\int_{h_i(s)}^s \right[ \lambda|y_i(\zeta)|
\\
& & \left.+e^{\lambda \tau_i}A_i|y_i(h_i(\zeta))|+\sum_{j=1}^me^{\lambda \sigma}L_{ij}|y_j(g_{ij}(\zeta))|+
\left( \sum_{j=1}^m L_{ij}e^{\lambda\sigma}+e^{\lambda \tau_i} A_i \right)\|\varphi_i\|\right] d\zeta
\\
& & \left.\left.+\sum_{j=1}^me^{\lambda \sigma}L_{ij}|y_j(g_{ij}(s))|+ \left( \sum_{j=1}^m L_{ij}+e^{\lambda \tau_i}A_i \right) \|\varphi_i\|\right)\right/\mu_i(s) \, ds,
\end{eqnarray*}
%\end{document}
where $\|\varphi_i\|=\sup_{t_0-\sigma\leq t\leq t_0} |\varphi_i(t)|$.
Let $\overline{y}_i=\max_{t_0\leq t\leq b}|y_i(t)|$.
If we fix some $b>t_0$ and denote $\overline{Y}=(\overline{y}_1,\dots,\overline{y}_n)^T$ then
\begin{eqnarray*}
\overline{y}_i & \leq & |x_i(t_0)|
+\frac{(\sum_{j=1}^m L_{ij}e^{\lambda\sigma}+e^{\lambda \tau_i }A_i)(A_ie^{\lambda \tau_i}\tau_i+1)}{\alpha_i-\lambda}\|\varphi_i\|
\\
& & \left. +\left[A_i\tau_ie^{\lambda \tau_i}
 \left(\lambda \overline{y}_i+A_ie^{\lambda \tau_i}\overline{y}_i+\sum_{j=1}^me^{\lambda \sigma}L_{ij}\overline{y}_j \right)
+\sum_{j=1}^me^{\lambda \sigma}L_{ij}\overline{y}_j \right]\right/(\alpha_i-\lambda) .
\end{eqnarray*}
We define the matrix $C(\lambda)=(c_{ij}(\lambda))_{i,j=1}^{m}$ with the entries
$$
c_{ii}(\lambda)=1-\frac{A_ie^{\lambda \tau_i}(\lambda+A_ie^{\lambda \tau_i}+e^{\lambda \sigma_{ii}}L_{ii})\tau_i+e^{\lambda \sigma_{ii}}L_{ii}}{\alpha_i-\lambda} ~,
$$
$$
c_{ij}(\lambda)=-\frac{A_ie^{\lambda \tau_i}e^{\lambda \sigma}L_{ij}\tau_i+e^{\lambda \sigma}L_{ij}}{\alpha_i-\lambda},~ i\neq j.
$$
Clearly, the vector inequality $C(\lambda) \overline{Y} \leq |X(t_0)|+M_1(\lambda)|\varphi|$ is valid for $t_0\leq t\leq b$, where
$$
M_1(\lambda)=\max_{1\leq i \leq m} \frac{(\sum_{j=1}^m L_{ij}e^{\lambda\sigma}+e^{\lambda \tau_i A_i})(A_ie^{\lambda \tau_i}\tau_i+1)}{\alpha_i-\lambda},
$$
and we have $\lim_{\lambda\rightarrow 0}C(\lambda)=C(0)=C$. By the assumption of the theorem, $C(0)$
is an M-matrix.
For $0<\lambda<\min_i \alpha_i$  the entries of the matrix $C(\lambda)$ are continuous functions; therefore, the determinant
of this matrix is a continuous function. For some small $\lambda >0$ all the principal minors of $C(\lambda)$ are positive;
the latter along with $c_{ij}(\lambda) \leq 0$, $i \neq j$ implies  that $C(\lambda)$ is an M-matrix for small $\lambda$.
If we fix such  parameter $\lambda=\lambda_0$,
then for $\overline{Y}$ there is an \textit{a priori} estimate
$$
\|\overline{Y}\|\leq M(\|X(t_0)\|+\|\varphi\|), ~~M=\|C^{-1}(\lambda_0)\| ~\max\{1,M_1(\lambda_0)\},
$$
where $M$ does not depend
on $b$ and $t_0$.
Finally, $X(t)=e^{-\lambda_0 (t-t_0)}Y(t)$,
hence
$$
\|X(t)\|\leq M \left( \|X(t_0)\|+\|\varphi\| \right) e^{-\lambda_0 (t-t_0)},
$$
which completes the proof.
\end{proof}

Consider the system with off-diagonal nonlinearities
\begin{equation}\label{1a}
\dot{x_i}(t)=-a_i(t)x_{i}(h_i(t))+\sum_{j\neq i}F_{ij}\left( t,x_j(g_{ij}(t)) \right),~t\geq 0,~ i=1,\dots,m.
\end{equation}
\begin{corol}\label{cor0}
Suppose that the matrix $C$ defined by
\begin{equation}\label{3b}
C=(c_{ij})_{i,j=1}^m, ~~c_{ii}=1-\frac{A_i^2\tau_i}{\alpha_i}, ~ c_{ij}=-\frac{A_iL_{ij}\tau_i+L_{ij}}{\alpha_i},i\neq j,
\end{equation}
is an M-matrix.
Then system (\ref{1a}) is globally exponentially stable.
\end{corol}

The next corollary examines the system with a non-delay linear term
\begin{equation}\label{4}
\dot{x_i}(t)=-a_i(t)x_{i}(t)+\sum_{j=1}^m F_{ij} \left( t,x_j(g_{ij}(t)) \right),~t\geq 0,~ i=1,\dots,m.
\end{equation}
\begin{equation}\label{5}
B=(b_{ij})_{i,j=1}^m,~~ b_{ii}=1-\frac{L_{ii}}{\alpha_i}, ~~ b_{ij}=-\frac{L_{ij}}{\alpha_i}, ~i\neq j.
\end{equation}

\begin{corol}\label{cor1}
Suppose $B$ defined by (\ref{5}) is an M-matrix. Then system (\ref{4}) is globally exponentially stable.
\end{corol}

For the delay linear system
\begin{equation}\label{6}
\dot{x_i}(t)=\sum_{j=1}^m a_{ij}(t)x_j(g_{ij}(t)), ~ i=1,\dots,m,
\end{equation}
 assume that $a_{ij}$ are essentially bounded on $[0,\infty)$ functions,
$0<\alpha_i\leq -a_{ii}(t)\leq A_i, |a_{ij}(t)|\leq A_{ij}, i\neq j$,
$g_{ij}$ are measurable functions, $0\leq t-g_{ij}(t)\leq \sigma_{ij}$.
Denote
\begin{equation}\label{7}
D=(d_{ij})_{i,j=1}^m,~~ d_{ii}=1-\frac{A_i^2\sigma_{ii}}{\alpha_i}, ~
d_{ij}=-\frac{A_iA_{ij}\sigma_{ii}+A_{ij}}{\alpha_i},~i\neq j.
\end{equation}

\begin{corol}\label{cor2}
Suppose $D$ defined by (\ref{7}) is an M-matrix. Then
system (\ref{6}) is exponentially stable.
\end{corol}

The same result holds for the linear system with non-delay diagonal terms
\begin{equation}\label{8}
\dot{x_i}(t)=-a_i(t)x_i(t)+\sum_{j \neq i} a_{ij}(t)x_j(g_{ij}(t)), ~ i=1,\dots,m,
\end{equation}
where $a_{ij}$ are essentially bounded on $[0,\infty)$ functions,
$0<\alpha_i\leq a_{i}(t)\leq A_i, |a_{ij}(t)|\leq A_{ij}, i\neq j$,
$g_{ij}$ are measurable functions, $0\leq t-g_{ij}(t)\leq \sigma$.
Denote
\begin{equation}\label{9}
 F=(f_{ij})_{i,j=1}^m, ~~f_{ii}=1, ~ f_{ij}=-\frac{A_{ij}}{\alpha_i},~i\neq j.
\end{equation}

\begin{corol}\label{cor3}
Suppose $F$  defined by (\ref{9}) is an M-matrix. Then
system (\ref{8}) is exponentially stable.
\end{corol}

\begin{corol}\label{cor4}
Suppose $m=2$,
$A_1(A_1+L_{11})\tau_1+L_{11} < \alpha_1$ and
\begin{equation}\label{10}
(\alpha_1-A_1(A_1+L_{11})\tau_1-L_{11})(\alpha_2-A_2(A_2+L_{22})\tau_2-L_{22})>L_{12}L_{21}
(1+A_1\tau_1)(1+A_2\tau_2).
\end{equation}
Then system (\ref{1}) is globally exponentially stable.
\end{corol}
\begin{proof}
For $m=2$ the matrix $C$ denoted by (\ref{3}) has the form
$$
C=
\left(\begin{array}{ll}
1-\frac{A_1(A_1+L_{11})\tau_1+L_{11}}{\alpha_1}& -\frac{A_1L_{12}\tau_1+L_{12}}{\alpha_1}\\ \vspace{2mm}
-\frac{A_2L_{21}\tau_2+L_{21}}{\alpha_2}& 1-\frac{A_2(A_2+L_{22})\tau_2+L_{22}}{\alpha_2}
\end{array}\right).
$$
The off-diagonal entries are negative, by the assumption of the corollary the principal minors are positive, so
$C$ is an M-matrix.
\end{proof}

\begin{corol}\label{cor5}
Suppose $m=2, L_{11}<\alpha_1$ and
 %\begin{equation}\label{11}
$(\alpha_1-L_{11})(\alpha_2-L_{22})>L_{12}L_{21}$.
%\end{equation}
Then system (\ref{4}) is globally exponentially stable.
\end{corol}

\begin{corol}\label{cor6}
Suppose $m=2$, $\sigma_{11} < \alpha_1/A_1^2$ and the inequality
% \begin{equation}\label{12}
$$
(\alpha_1-A_1^2\sigma_{11})(\alpha_2-A_2^2\sigma_{22})>A_{12}A_{21}
(1+A_1\sigma_{11})(1+A_2\sigma_{22})
$$
%\end{equation}
holds. Then system (\ref{6}) is  exponentially stable.
\end{corol}

\begin{corol}\label{cor7}
Suppose $m=2$ and
% \begin{equation}\label{13}
$\alpha_1\alpha_2>A_{12}A_{21}$.
%\end{equation}
Then system (\ref{8})
is exponentially stable.
\end{corol}

\begin{remark}\label{remark1}
By Corollary \ref{cor7}, the linear system
 \begin{equation}\label{14}
 \begin{array}{l}
\dot{x}(t)=-a_{11}x(t)+a_{12}y(t)\\
\dot{y}(t)= a_{21}x(t)-a_{22}y(t)\\
\end{array}
\end{equation}
with the coefficients $a_{ii}>0$ is exponentially stable if
\begin{equation}\label{15}
a_{11}a_{22}>a_{12}a_{21}.
\end{equation}
Condition (\ref{15}) is necessary and sufficient for exponential stability
of system (\ref{14}); therefore, Theorem \ref{th1} and its
corollaries cannot be improved.
\end{remark}

In the paper \cite{Gop} Gopalsamy considered autonomous system \eqref{02}. For $n=1$ it has the form
\begin{equation}\label{16}
\begin{array}{l}
\dot{x}(t)=-a_1x(t-\tau_1)+a_{12}f_1(y(t-\sigma_1))\\
\dot{y}(t)=-a_2y(t-\tau_2)+a_{21}f_2(x(t-\sigma_2))\\
\end{array}
\end{equation}
where $a_i>0, a_{ij}>0, \tau_i\geq 0, \sigma_i\geq 0, |f_i(u)|\leq L_i |u|~ and~ i=1,2$.
In \cite{Gop} the following global attractivity result was obtained. If $a_i\tau_i<1$ and
\begin{equation}\label{17}
\frac{1-a_1\tau_1}{1+a_1\tau_1}>\frac{a_{12}L_1}{a_1},~
\frac{1-a_2\tau_2}{1+a_2\tau_2}>\frac{a_{21}L_2}{a_2}
\end{equation}
then any solution of system (\ref{16}) tends to zero.
By Corollary \ref{cor4} equation  (\ref{16}) is exponentially stable if  $a_i\tau_i<1$ and
\begin{equation}\label{18}
\frac{(1-a_1\tau_1)(1-a_2\tau_2)}{(1+a_1\tau_1)(1+a_2\tau_2)}
>\frac{a_{12}a_{21}L_1L_2}{a_1a_2}.
\end{equation}
Obviously condition (\ref{17}) implies (\ref{18}).

\begin{example}
\label{example1}
Consider system (\ref{16}) where $a_1=0.8$, $a_2=0.5$, $a_{12}=a_{21}=1$,  $\tau_1=0.5$, $\tau_2=0.4$,
$|f_i(u)|\leq L_i |u|$ with $L_1=0.5$, $L_2=0.2$, $\sigma_i\geq 0$. Here the first inequality in (\ref{17}) does not hold
since
$$\frac{1-a_1\tau_1}{1+a_1\tau_1}=\frac{3}{7}< \frac{5}{8}=\frac{a_{12}L_1}{a_1},
$$
and therefore the result of \cite{Gop} cannot be applied. However, $a_1\tau_1=0.4<1$ and inequality (\ref{18})
$$
\frac{(1-a_1\tau_1)(1-a_2\tau_2)}{(1+a_1\tau_1)(1+a_2\tau_2)}
=\frac{2}{7}>\frac{1}{4}=\frac{a_{12}a_{21}L_1L_2}{a_1a_2}$$
holds, thus Corollary~\ref{cor4} implies exponential stability, hence for $n=1$ ($m=2$) we obtained the result which is sharper
than the relevant result in \cite{Gop}.
\end{example}

In the next section, we provide more in-depth analysis of systems with leakage delays which include \eqref{16} as a special case.

\section{BAM Network with  Time-Varying Delays}
\label{BAMsection}

In \cite{Gop} a class (\ref{02}) of BAM neural networks with leakage (forgetting) delays  was under study.
Via Lyapunov functionals method  sufficient conditions for the existence of a unique equilibrium and its global stability for system
(\ref{02}) were obtained. To extend and improve the results obtained in \cite{Gop}
and  \cite{Liu,Peng}, we will focus on the non-autonomous BAM neural network model
\begin{equation}\label{22}
\begin{array}{l}
{\displaystyle \frac{dx_i(t)}{dt}=r_i(t)\left(-a_ix_i(h_i^{(1)}(t))+\sum_{j=1}^n a_{ij}f_j \left(y_j \left( l_j^{(2)}(t)\right) \right)+I_i\right) } \\
{\displaystyle \frac{dy_i(t)}{dt}=p_i(t)\left(-b_iy_i(h_i^{(2)}(t))+\sum_{j=1}^n b_{ij}g_j \left(x_j \left( l_j^{(1)}(t)\right) \right)+J_i\right), }
\end{array}
\end{equation}
$i=1,\dots,n$, $t\geq 0$,
with the initial conditions
\begin{equation}\label{22a}
x_i(t)=\varphi_i(t), ~~y_i(t)=\varphi_{i+n}(t),~~t<0, ~~i=1,\dots,n.
\end{equation}

The following auxiliary lemma will be used.

\begin{lemma}\label{lemma1}
Let
\begin{equation}\label{20}
u_i=\sum_{j=1}^m F_{ij}(u_{j}),~i=1,\dots,m,
\end{equation}
where $|F_{ij}(u)-F_{ij}(v)|\leq L_{ij}|u-v|$,
the matrix $L$ be $L= (L_{ij})_{i,j=1}^m$ and denote by $r(L)$ the spectral radius of $L$.
If $r(L)<1$ then system  (\ref{20}) has a unique solution.
\end{lemma}
\begin{proof}
Consider the operator $T: \RR^m\rightarrow \RR^m$ denoted by
$$
T(u):= T( (u_1,\dots,u_m)^T)=\left( \sum_{j=1}^m F_{1j}(u_{j}),\dots,\sum_{j=1}^m F_{mj}(u_{j}) \right)^T.
$$
Then
$$
|T(u)-T(v)|\leq \left( \sum_{j=1}^m L_{1j}|u_{j}-v_{j}|,\dots,\sum_{j=1}^m L_{mj}|u_{j}-v_{j}| \right)^T=L|u-v|.
$$
It is well known that $r(L)=\inf_{\|\cdot\|}\|L\|$, where the infimum is taken on all (equivalent) norms in $\RR^m$.
Since $r(L)<1$, we can choose a norm in $\RR^m$ such that the corresponding norm $\|L\|\leq q<1$. We fix now such a norm
and have
$$
\|T(u)-T(v)\|\leq \|L\|\|u-v\|\leq q<1.
$$
By the Banach contraction principle the equation $u=T(u)$ has a unique solution.
\end{proof}

\begin{corol}\label{cor8}
Suppose at least one of the following conditions holds:\\
1. $\max |\lambda(L)|<1$, where the maximum is taken over all eigenvalues of matrix $L$.\\
2. $\max_{i} \sum_{j=1}^m L_{ij}<1$.\\
3. $\max_{j} \sum_{i=1}^m L_{ij}<1$.\\
4. $\sum_{i=1}^m\sum_{j=1}^m L^2_{ij}<1$.\\
Then system  (\ref{20}) has a unique solution.
\end{corol}

It should be noted that the proof of Lemma \ref{lemma1} is original and
shorter than, for example, the recently published proof \cite[Theorem 2.2]{Wong}.

Henceforth, assume that the following assumptions hold for (\ref{22}), (\ref{22a}):
\\
(b1) $r_i, p_i$ are Lebesgue measurable essentially bounded  on $[0,\infty)$ functions, $0<\alpha_i\leq r_i(t)\leq R_i, 0<\beta_i\leq p_i(t)\leq P_i$;\\
(b2) $f_j(\cdot), g_j(\cdot)$ are continuous functions;
$|f_j(u)-f_j(v)|\leq L_j^f|u-v|, |g_j(u)-g_j(v)|\leq L_j^g|u-v|$;\\
(b3) $h_i^{(1)}, h_i^{(2)}, l_j^{(1)}, l_j^{(2)}$ are Lebesgue measurable functions,
$
0\leq t-h_i^{(1)}(t)\leq \tau_i^{(1)},~ 0\leq t-h_i^{(2)}(t)\leq \tau_i^{(2)}$,\\
$0\leq t-l_i^{(1)}(t)\leq \sigma_i^{(1)},~ 0\leq t-l_i^{(2)}(t)\leq \sigma_i^{(2)};
 $ \\
(b4) $\varphi_i$ are continuous functions.

Let
$
(x^\ast,y^\ast)= \left( x_1^\ast,\dots,x_n^\ast,y_1^\ast,\dots,y_n^\ast\right)
$
be a solution of the system
\begin{equation}\label{23}
\begin{array}{l}
{\displaystyle a_ix_i=\sum_{j=1}^na_{ij}f_j(y_j)+I_i} \\
{\displaystyle b_iy_i=\sum_{j=1}^nb_{ij}g_j(x_j)+J_i.}
\end{array}
\end{equation}
Apparently the existence of a solution of system (\ref{23}) is equivalent
to the existence of the solution of the following system
\begin{equation}\label{24}
\begin{array}{l}
{\displaystyle u_i=\sum_{j=1}^na_{ij}f_j\left(\frac{v_j}{b_j}\right)+I_i}\\
{\displaystyle v_i=\sum_{j=1}^nb_{ij}g_j\left(\frac{u_j}{a_j}\right)+J_i.}
\end{array}
\end{equation}
Denoting $u_j=x_j,~j=1,\dots, n; u_j=y_{j-n},~ j=n+1,\dots, 2n$,
$$
F_{ij}(u)=\left\{\begin{array}{ll}
\frac{a_{i,j-n}}{a_i}f_{j-n}(u)+\frac{I_i}{a_i},& i=1,\dots,n; j=n+1,\dots,2n,\\
\frac{b_{i-n,j}}{b_{i-n}}g_j(u)+\frac{J_{i-n}}{a_{i-n}},& i=n+1,\dots,2n; j=1,\dots,n,\\
0,& otherwise,
\end{array}\right.
$$
we can rewrite system (\ref{23}) in the form of (\ref{20}) with $m=2n$, $|F_{ij}(u)-F_{ij}(v)|\leq L_{ij}|u-v|$.

We introduce the matrix $A=(L_{ij})_{i,j=1}^{2n}$
$$
A=\left(\begin{array}{llllll}
0&\dots&0&\frac{|a_{11}|L_1^f}{a_1}&\dots&\frac{|a_{1n}|L_n^f}{a_1}\\
-&-&-&-&-&-\\
0&\dots&0&\frac{|a_{n1}|L_1^f}{a_n}&\dots&\frac{|a_{nn}|L_n^f}{a_n}\\
\frac{|b_{11}|L_1^g}{b_1}&\dots&\frac{|b_{1n}|L_n^g}{b_1}&0&\dots&0\\
-&-&-&-&-&-\\
\frac{|b_{n1}|L_1^g}{b_n}&\dots&\frac{|b_{nn}|L_n^g}{b_n}&0&\dots&0
\end{array}\right).
$$
By the same token we can rewrite  (\ref{24}) in the form (\ref{20}), where
$$
F_{ij}(u)=\left\{\begin{array}{ll}
a_{i,j-n}f_{j-n}(\frac{u}{b_{j-n}})+I_i,& i=1,\dots,n; j=n+1,\dots,2n,\\
b_{i-n,j}g_j(\frac{u}{a_j})+J_{i-n},& i=n+1,\dots,2n; j=1,\dots,n,\\
0,& otherwise,
\end{array}\right.
$$
with $m=2n$, $|F_{ij}(u)-F_{ij}(v)|\leq L_{ij}|u-v|$, and introduce the matrix $B=(L_{ij})_{i,j=1}^{2n}$
$$
B=\left(\begin{array}{llllll}
0&\dots&0&\frac{|a_{11}|L_1^f}{b_1}&\dots&\frac{|a_{1n}|L_n^f}{b_n}\\
-&-&-&-&-&-\\
0&\dots&0&\frac{|a_{n1}|L_1^f}{b_1}&\dots&\frac{|a_{nn}|L_n^f}{b_n}\\
\frac{|b_{11}|L_1^g}{a_1}&\dots&\frac{|b_{1n}|L_n^g}{a_n}&0&\dots&0\\
-&-&-&-&-&-\\
\frac{|b_{n1}|L_1^g}{a_1}&\dots&\frac{|b_{nn}|L_n^g}{a_n}&0&\dots&0
\end{array}\right).
$$
In the following theorem we apply Lemma~\ref{lemma1} to systems (\ref{23}) and (\ref{24}) with $L=A$ and $L=B$, and obtain
conditions $1-4$ and $5-8$, respectively.

\begin{theorem}\label{th2}
Suppose at least one of the following conditions holds:

1. $\max |\lambda(A)|<1$, where maximum is taken on all eigenvalues of matrix $A$.

2. ${\displaystyle \max_i \sum_{j=1}^n \frac{|a_{ij}|L_j^f}{a_i}<1,~~ \max_i \sum_{j=1}^n \frac{|b_{ij}|L_j^g}{b_i}<1}$.

3. ${\displaystyle \max_j \sum_{i=1}^n \frac{|a_{ij}| L_j^f}{a_i}<1,~~ \max_j \sum_{i=1}^n \frac{|b_{ij}| L_j^g}{b_i}<1}$.

4. ${\displaystyle \sum_{i=1}^n \sum_{j=1}^n \left[  \left(\frac{|a_{ij}|L_j^f}{a_i}\right)^2+\left(\frac{|b_{ij}|L_j^g}{b_i}\right)^2 \right] <1 }$.

5. $\max |\lambda(B)|<1$, where maximum is taken on all eigenvalues of matrix $B$.

6. ${\displaystyle \max_i \sum_{j=1}^n \frac{|a_{ij}|L_j^f}{b_j}<1, ~~\max_i \sum_{j=1}^n \frac{|b_{ij}|L_j^g}{a_j}<1}$.

7. ${\displaystyle \max_j \sum_{i=1}^n \frac{|a_{ij}|L_j^f}{b_j}<1, ~~\max_j \sum_{i=1}^n \frac{|b_{ij}|L_j^g}{a_j}<1}$.

8. ${\displaystyle  \sum_{i=1}^n \sum_{j=1}^n  \left[ \left(\frac{|a_{ij}|L_j^f}{b_j}\right)^2+\left(\frac{|b_{ij}|L_j^g}{a_j}\right)^2 \right]  <1 }$.

Then system  (\ref{23}) has a unique solution and thus  system (\ref{22}) has a unique equilibrium.
\end{theorem}

\begin{remark}\label{rm1}
Note that the conclusion of Theorem~\ref{th2} under condition 7 was obtained in paper \cite{Gop}.
\end{remark}

Below, assume that system (\ref{22}) has a unique equilibrium $(x^{\ast},y^{\ast})$.
To obtain a global stability condition for this equilibrium, consider the matrix $C_{BAM}=(c_{ij})_{i,j=1}^{2n}$, where
 \begin{equation}\label{25}
c_{ii}=\left\{\begin{array}{ll}
1-a_iR_i^2\tau_i^{(1)}/\alpha_i,& i=1,\dots,n,\\
1-b_{i-n}P_{i-n}^2\tau_{i-n}^{(2)}/\beta_{i-n},& i=n+1,\dots,2n,
\end{array}\right.
\end{equation}
\begin{equation}\label{26}
c_{ij}=\left\{\begin{array}{l}
-|a_{i,j-n}|R_iL_{j-n}^f(a_iR_i\tau_i^{(1)}+1)/(\alpha_i a_i),~
i=1,\dots,n,j=n+1,\dots,2n,\\
-|b_{i-n,j}|P_{i-n}L_j^g(b_{i-n}P_{i-n}\tau_{i-n}^{(2)}+1)/(\beta_{i-n} b_{i-n}),\,
i=n+1,\dots,2n,j=1,\dots,n,~
\\ 0,~ otherwise.
\end{array}\right.
\end{equation}

\begin{theorem}\label{th3}
Suppose  matrix $C_{BAM}$ is an M-matrix.
Then the equilibrium $(x^\ast,y^\ast)$ of system (\ref{22}) is globally exponentially stable.
\end{theorem}
\begin{proof}
After the substitution $x_i(t)=u_i(t)+x_i^\ast, y_i(t)=v_i(t)+y_i^\ast$,
system (\ref{22}) has the form
\begin{equation}\label{27}
\begin{array}{l}
{\displaystyle \dot{u_i}(t)=-r_i(t)a_iu_i(h_i^{(1)}(t))+\sum_{j=1}^n
a_{ij}r_i(t)\left(f_j(v_j(l_j^{(2)}(t))+y_j^{\ast})-f_j(y_j^{\ast})\right)} \\ \\
{\displaystyle \dot{v_i}(t)=-p_i(t)b_iv_i(h_i^{(2)}(t))+\sum_{j=1}^n
b_{ij}p_i(t)\left(g_j(u_j(l_j^{(1)}(t))+x_j^{\ast})-g_j(x_j^{\ast})\right),}
\end{array}
\end{equation}
$$
x_i(t)=\left\{\begin{array}{ll}
u_i(t),& i=1,\dots,n,\\
v_{i-n}(t),& i=n+1,\dots, 2n,
\end{array}\right.
~~
a_i(t)=\left\{\begin{array}{ll}
r_i(t)a_i,& i=1,\dots,n,\\
p_{i-n}(t)b_{i-n},& i=n+1,\dots, 2n,
\end{array}\right.
$$$$
h_i(t)=\left\{\begin{array}{ll}
h_i^{(1)}(t),& i=1,\dots,n,\\
h_{i-n}^{(2)}(t),& i=n+1,\dots, 2n,
\end{array}\right.
$$ $$
g_{i,j}(t)=\left\{\begin{array}{ll}
l_{j-n}^{(2)}(t),& i=1,\dots,n, j=n+1,\dots,2n,\\
l_j^{(1)}(t),& i=n+1,\dots, 2n,j=1,\dots,n,\\
0,& otherwise,
\end{array}\right.
$$
$$
F_{i,j}(t,x)=\left\{\begin{array}{l}
a_{i,j-n}r_i(t)\left(f_{j-n}(x+y_{j-n}^{\ast})-f_{j-n}(y_{j-n}^{\ast})\right),~ i=1,\dots,n, j=n+1,\dots,2n,\\
b_{i-n,j}p_{i-n}(t)\left(g_j(x+x_j^{\ast})-g_j(x_j^{\ast})\right),~ i=n+1,\dots, 2n,j=1,\dots,n,\\
0,~ otherwise.
\end{array}\right.
$$
We have $0<\alpha_i\leq a_i(t)\leq A_i$, where
$$
\alpha_i=\left\{\begin{array}{ll}
r_ia_i,&i=1,\dots,n,\\
p_{i-n}b_{i-n},& i=n+1,\dots,2n,
\end{array}\right. ~~~
A_i=\left\{\begin{array}{ll}
R_ia_i,&i=1,\dots,n,\\
P_{i-n}b_{i-n},& i=n+1,\dots,2n,
\end{array}\right.
$$
and $|F_{ij}(t,x)|\leq L_{ij}|x|$ with the constant
$$
L_{i,j}=\left\{\begin{array}{ll}
|a_{i,j-n}|R_iL_{j-n}^f,& i=1,\dots,n, j=n+1,\dots,2n,\\
|b_{i-n,j}|P_{i-n}L_j^{g},& i=n+1,\dots, 2n,j=1,\dots,n,\\
0,& otherwise.
\end{array}\right.
$$
System (\ref{27}) with $m=2n$  has form (\ref{1a}), where matrix $C_{BAM}$  corresponds
to matrix $C$ defined by (\ref{3b}). All conditions of
Corollary~\ref{cor0} hold; therefore, the trivial solution of system (\ref{27})
is globally exponentially stable; hence the equilibrium $(x^\ast,y^\ast)$ of system (\ref{22}) is globally exponentially stable.
\end{proof}

\begin{corol}\label{cor9}
Suppose at least one of the following conditions holds: \\
1.
${\displaystyle \sum_{j=1}^n \frac{|a_{ij}|R_iL_j^f(a_iR_i\tau_i^{(1)}+1)}{\alpha_i a_i}<
1-\frac{a_iR_i^2\tau_i^{(1)}}{\alpha_i}}$,\\
${\displaystyle \sum_{j=1}^n \frac{|b_{ij}|P_iL_j^g(b_iP_i\tau_i^{(2)}+1)}{\beta_i b_i}<
1-\frac{b_iP_i^2\tau_i^{(2)}}{\beta_i},}$ ~
$i=1,\dots,n.$ \\ \\
2.
${\displaystyle \sum_{i=1}^n \frac{|a_{ij}|R_iL_j^f(a_iR_i\tau_i^{(1)}+1)}{\alpha_i a_i}<
1-\frac{b_jP_j^2\tau_j^{(2)}}{\beta_j}}$, \\
${\displaystyle \sum_{i=1}^n \frac{|b_{ij}|P_iL_j^g(b_iP_i\tau_i^{(2)}+1)}{\beta_i b_i}<
1-\frac{a_jR_j^2\tau_j^{(1)}}{\alpha_j},}$ ~
$j=1, \dots, n$.
\\
3. There exist positive numbers $\mu_k, k=1,\dots, 2n$ such that
$$
 \sum_{j=1}^n \frac{\mu_{j+n} |a_{ij}|R_iL_j^f(a_iR_i\tau_i^{(1)}+1)}{\alpha_i a_i}<
\mu_i\left(1-\frac{a_iR_i^2\tau_i^{(1)}}{\alpha_i}\right),$$ $$
\sum_{j=1}^n \frac{\mu_{j}|b_{ij}|P_iL_j^g(b_iP_i\tau_i^{(2)}+1)}{\beta_i b_i}<
\mu_{i+n}\left(1-\frac{b_iP_i^2\tau_i^{(2)}}{\beta_i}\right),$$
$(i=1,\dots,n).$
\\
4. There exist positive numbers $\mu_k, k=1,\dots, 2n$ such that
$$
\sum_{i=1}^n \frac{\mu_{i+n}|a_{ij}|R_iL_j^f(a_iR_i\tau_i^{(1)}+1)}{\alpha_i a_i}<
\mu_j\left(1-\frac{b_jP_j^2\tau_j^{(2)}}{\beta_j}\right),$$$$
\sum_{i=1}^n \frac{\mu_{j}|b_{ij}|P_iL_j^g(b_iP_i\tau_i^{(2)}+1)}{\beta_i b_i}<
\mu_{j+n}\left(1-\frac{a_jR_j^2\tau_j^{(1)}}{\alpha_j}\right),
$$
$(j=1, \dots, n$).

Then the equilibrium $(x^\ast,y^\ast)$ of system (\ref{22}) is globally exponentially stable.
\end{corol}
\begin{proof}
By Lemma~\ref{lemma0} any of the conditions $1-4$  implies
 that $C_{BAM}$ is an M-matrix.
\end{proof}

\begin{remark}
Part 3 of Corollary~\ref{cor9} coincides with \cite[Theorem 3.1]{Liu} in the case when $r_i(t)$ and $p_i(t)$ are constants.
In addition to being more general than \cite[Theorem 3.1]{Liu}, the result of Theorem~\ref{th3} does not require to find some
positive constants, i.e., the check of the signs of principal minors will immediately indicate whether such constants exist or not.
\end{remark}

In the following statement consider system (\ref{22}) without delays in the leakage terms.
\begin{corol}\label{cor10}
Suppose  $h_i^{(1)}(t)\equiv t, h_i^{(2)}(t)\equiv t$,
and  at least one of the following conditions holds:
\\
1. ${\displaystyle \sum_{j=1}^n \frac{|a_{ij}|R_iL_j^f}{\alpha_i a_i}<1,~~
\sum_{j=1}^n \frac{|b_{ij}|P_iL_j^g}{\beta_i b_i}<1},~~
(i=1, \dots, n).
$
\\
2. ${\displaystyle \sum_{i=1}^n \frac{|a_{ij}|R_iL_j^f}{\alpha_i a_i}<1,~~
\sum_{i=1}^n \frac{|b_{ij}|P_iL_j^g}{\beta_i b_i}<1}, ~~(j=1, \dots, n).
$
\\
3. There exist positive numbers $\mu_k, k=1,\dots, 2n$ such that \\
${\displaystyle \sum_{j=1}^n \frac{\mu_{j+n}|a_{ij}|R_iL_j^f}{\alpha_i a_i}<\mu_i,~~
\sum_{j=1}^n \frac{\mu_{j}|b_{ij}|P_iL_j^g}{\beta_i b_i}<\mu_{i+n}},~~
(i=1, \dots, n).
$ \\
4. There exist positive numbers $\mu_k, k=1,\dots, 2n$ such that \\
${\displaystyle \sum_{i=1}^n \frac{\mu_{i+n}|a_{ij}|R_iL_j^f}{\alpha_i a_i}<\mu_j,~~
\sum_{i=1}^n \frac{\mu_{i}|b_{ij}|P_iL_j^g}{\beta_i b_i}<\mu_{j+n},} ~~(j=1, \dots, n).
$

Then the equilibrium $(x^\ast,y^\ast)$ of system (\ref{22}) is globally exponentially stable.
\end{corol}

Consider system (\ref{22}) with $n=1$:
\begin{equation}\label{27a}
\begin{array}{l}
\frac{dx}{dt}=r(t)\left(-ax(h_1(t))+Af(y(l_2(t)))+I\right)\\\\
\frac{dy}{dt}=p(t)\left(-by(h_2(t))+Bg(x(l_1(t)))+J\right)
\end{array}
\end{equation}
where
$$
a>0,~ b>0,~ 0<\alpha\leq r(t)\leq R,~ 0<\beta\leq p(t)\leq P,~ |f(u)-f(v)|\leq L^f|u-v|,
$$$$
|g(u)-g(v)|\leq L^g|u-v|, ~ 0\leq t-h_i(t)\leq \tau_i, ~ 0\leq t-l_i(t)\leq \sigma_i,~ i=1,2.
$$

\begin{corol}\label{cor11}
Suppose
${\displaystyle
\frac{aR^2\tau_1}{\alpha}<1},$
and
$$
\frac{ABRPL^fL^g(aR\tau_1+1)(aP\tau_2+1)}{\alpha \beta ab}
<\left(1-\frac{aR^2\tau_1}{\alpha}\right)\left(1-\frac{bP^2\tau_2}{\beta}\right).
$$
Then the equilibrium
$
(x^\ast,y^\ast)
$
of system (\ref{27a}) is globally exponentially stable.
\end{corol}

\begin{example}
Consider the particular case of BAM network described by (\ref{27a})
\begin{equation}\label{27a_part}
\begin{array}{l}
{\displaystyle \frac{dx}{dt}=(20+\mu \sin t)\left[-x\left(t - \frac{1+|\sin t|}{2000}\right)+ \frac{1}{720} y \left(t-3\sin^2(t)\right)+ 10000 \right] }\\ \\
{\displaystyle \frac{dy}{dt}=(40+\mu \cos t) \left[-x\left(t - \frac{1+|\cos t|}{2000}\right)+ \frac{1}{200} y \left(t-2\sin^2(t)\right)+ 20000 \right] }
\end{array}
\end{equation}
for $\mu\geq 0$.
Here $L^f=L^g=a=b=1$, $\alpha=20-\mu$, $R=20+\mu$, $\beta=4-\mu$, $P=40+\mu$, $\tau_1=\tau_2=\frac{1}{1000}$, $A=\frac{1}{720}$, $B=\frac{1}{200}$.

By Corollary~\ref{cor11}, system (\ref{27a_part}) is exponentially stable if
${\displaystyle  \frac{(20+\mu)^2}{1000(20-\mu)}<1}$ and
$$ \frac{(\frac{20 + \mu}{1000} +   1)(\frac{40 + \mu}{1000} + 1)}{720 \cdot 200(20 - \mu)(40 - \mu)} <
\left(1 - \frac{(20 + \mu)^2}{1000(20 - \mu)} \right)\left(1 - \frac{(40 + \mu)^2}{1000(40 - \mu)} \right),
$$
which is satisfied, for example, if $0\leq \mu \leq 18$.

Note that for $\mu=0$ this example coincides with \cite[Example 4.1]{Liu}.
It was also mentioned in \cite{Liu} that exponential stability of (\ref{27a_part}) cannot be obtained using the results of
\cite{Gop,Li2,Li3,Wang2}, since leakage delays in (\ref{27a_part}) are time-variable.
Therefore, our results for the case $\mu>0$ and with time-varying coefficients and delays
are new and  applicable to more general models.
\end{example}

\section{Discussion and Open Problems}
\label{discussion}

To obtain sufficient stability conditions for nonlinear delay systems four different approaches might be used: construction of Lyapunov functionals, application of fixed point theory, either development of estimations for matrix or operator norms,  or making use of some special matrix ($M-$matrix) properties,
and the transformations of a given nonlinear system to an operator equation with a Volterra casual operator.  While the Lyapunov direct method has been and remains the leading technique, numerous difficulties with the theory and  applications to specific systems persist. One of the problems with using the fixed point techniques is the construction of an appropriate map (integral equation) that is sometimes quite difficult or impossible.
The technique used in papers \cite{Li2}--\cite{Lili} and \cite{Xio} is based on the construction of Lyapunov functionals.

Perhaps it is worth mentioning that the method applied here to nonlinear system (\ref{1}) is somehow related to
 the approach used in \cite{Gy} for linear system (\ref{6}); however, to the best of our knowledge,
there are no similar results for (\ref{1}).
Remark \ref{remark1},  Example \ref{example1},
Theorem \ref{th3} and its corollaries improve and extend results previously obtained
for BAM neural networks in  \cite{Chen, Gop, Liu, Wang2, Zen, Zen2}.

In the framework of this paper, we could not consider all the applications of the $M$-matrix method to specific models;
therefore we outline some particular cases and extensions that might be of interest

 for scientists who plan
to start future research in this field.

\begin{enumerate}
\item
Find global stability conditions of system (\ref{1})
 for the special cases: $$F_{ij} (t,x)=\tanh\left(\alpha_i x(g_{ij} (t))\right) \mbox{~~and~~}
F_{ij} (t,x)=\frac{1}{1+\alpha_i e^{-x(g_{ij} (t))}}.\nonumber$$
\item
Study global stability  for a more general than (\ref{1}) model:
\begin{equation}\label{1add}
\dot{x_i}(t)=-a_i(t)x_{i}(h_i(t))+\sum_{j=1}^m\sum_{k=1}^sF_{ijk}\left( t,x_j(g_{ijk}(t)) \right),~t\geq 0,~ i=1,\dots,m.\nonumber
\end{equation}
\item
Derive sufficient stability tests for the equation with a distributed transmission delay
\begin{equation}\label{1_dist_int}
\dot{x_i}(t)=-a_i(t)x_{i}(h_i(t))+\sum_{j=1}^m \int_{t-\tau_j}^t K_{ij}(t,s) F_{ij}\left( s,x_j(g_{ij}(s)) \right)~ds,
\nonumber
\end{equation}
$t\geq 0,~ i=1,\dots,m$.
\item
Investigate stability of the system with distributed delays in all terms
\begin{equation}\label{2_dist_gen}
\dot{x_i}(t)=-a_i(t) \int_{t-\sigma_i}^t \!\!\!\!\! x_{i}(h_i(s))~d_s R_i(t,s) +\sum_{j=1}^m \int_{t-\tau_j}^t
\!\!\!\!\! F_{ij}\left(
s,x_j(g_{ij}(s)) \right)~d_s T_{ij}(t,s),\nonumber
\end{equation}
$t\geq 0,~ i=1,\dots,m$.
\item
Obtain sufficient stability conditions for the system with an infinite distributed delay
\begin{equation}\label{3_dist_int}
\dot{x_i}(t)=-a_i(t)x_{i}(h_i(t))+\sum_{j=1}^m \int_{-\infty}^t K_{ij}(t,s) F_{ij}\left( s,x_j(g_{ij}(s)) \right)~ds,
\nonumber
\end{equation}
$t\geq 0$, $i=1,\dots,m$,
where $|K_{ij}(t,s)| \leq Me^{-\nu(t-s)}$. Generalize this result to the case of exponentially decaying infinite leakage delays as well.
\item
Analyze global asymptotic stability conditions of (\ref{1}) when condition\nonumber(a3) is not satisfied but
$\lim_{t\to\infty} h_i(t)=\infty$, $\lim_{t\to\infty} g_{ij}(t)=\infty$, e.g., the pantograph-type delays $h_i(t)=\lambda_i(t)$
for $0<\lambda_i<1$.
Is it possible to estimate the rate of convergence for some classes of delays?
\item
Under which conditions will solutions of BAM system (\ref{22})
with $I_i>0$, $J_i>0$ and positive initial functions be permanent (positive, bounded and separated from zero)?
\item
Apply the $M$-matrix method to the following generalization of  (\ref{1})
$$
\dot{x_i}(t)=-a_i(t)x_{i}(h_i(t))+\sum_{j=1}^m F_{ij}\left( t,x_1(g_{ij}^{(1)}(t)),\dots, x_m(g_{ij}^{(m)}(t)) \right),
$$
$t\geq 0,~ i=1,\dots,m$.
\item
Conjecture:\\
If $C$ defined by (\ref{3}) has negative off-diagonal entries $a_{ij} \leq 0$, $i \neq j$,
and its Moore-Penrose pseudoinverse matrix is nonnegative then system (\ref{1}) is stable.
\end{enumerate}

\begin{remark}
To apply the results of the present paper, first use \cite[Theorem 9]{MCM2008}
to reduce exponential stability of equations with distributed delays
to exponential stability of  equations with concentrated delays.
For some other methods see recent papers \cite{BB1,BB2,BB3} and references therein.
\end{remark}

\end{document}